\documentclass[12pt,a4paper]{article}

\usepackage{amsmath,amsthm,amssymb,amscd,a4wide}

\usepackage{dsfont}
\usepackage{mathrsfs}
\usepackage{setspace}
\usepackage{hyperref}

\usepackage{xcolor}

\newcommand{\ud}{\mathrm{d}}

\newcommand{\eins}{\mathds{1}}

\numberwithin{equation}{section}

\newtheorem{theorem}{Theorem}[section]
\newtheorem{lemma}[theorem]{Lemma}
\newtheorem{prop}[theorem]{Proposition}

\theoremstyle{definition}

\numberwithin{equation}{section}

\begin{document}

\thispagestyle{empty}

\vspace*{1cm}

\begin{center}

{\LARGE\bf A lower bound on the spectral gap of Schrödinger operators with weak potentials of compact support} \\

\vspace*{2cm}

{\large Joachim Kerner \footnote{E-mail address: {\tt Joachim.Kerner@fernuni-hagen.de}}}%

\vspace*{5mm}

Department of Mathematics and Computer Science\\
FernUniversit\"{a}t in Hagen\\
58084 Hagen\\
Germany\\

\end{center}

\vfill

\begin{abstract} In this paper we continue the study of the spectral gap of Schrödinger operators on large intervals and subject to Neumann boundary conditions. The main goal is to derive a lower bound on the spectral gap which is polynomial in the interval length. This bound is derived for a class of bounded potentials of compact support which are weak enough in a suitable sense.
\end{abstract}

\newpage


%
\section{Introduction}
In this paper we aim at estimating the spectral gap (or gap for short) of a Schrödinger operator on an interval of length $L > 0$ from below in the limit of $L$ going to infinity. The spectral gap, on the other hand, is defined as the distance between the first two eigenvalues and constitutes a classical quantity in the spectral theory of operators. For example, in \cite{AB,Abramovich,Lavine} and more recently in \cite{AK,ACH}, the spectral gap of the Laplacian (subject to certain self-adjoint boundary conditions) on a fixed interval was compared with the spectral gap of the Laplacian plus some additional potential on the same interval. It turns out that the spectral gap may increase or decrease, depending on specific properties of the potential considered; for example, as shown in \cite{Lavine}, the spectral gap for the (Dirichlet-) Laplacian always increases given the added potential is convex. Of course, a generic potential usually fails, e.g., to be convex and hence it seems rather difficult to control the spectral gap for a Schrödinger operator on a fixed interval. 

However, as discussed in the recent paper \cite{KT20}, the asymptotic behaviour of the gap can nevertheless be studied for a rather general class of potentials. A main finding of \cite{KT20} was that the spectral gap converges to zero strictly faster than the spectral gap of the free (Dirichlet-) Laplacian for potentials that decay fast enough at infinity. This holds, in particular, for potentials of compact support. Furthermore, in \cite{KT20} the authors also put forward the conjecture that the spectral gap cannot close faster than $\sim L^{-3}$ for bounded compactly supported potentials and this is indeed what motivates the present paper. More explicitly, we aim at providing a lower bound to the spectral gap of a Schrödinger operator for certain compactly supported potentials which is polynomial in the interval length. As a matter of fact, for the potentials considered, we are able to prove that the spectral gap cannot close faster than $\sim L^{-4}$. At this point it is important to stress that an explicit lower bound for the spectral gap was derived recently in \cite{KernerLowerBound} for a much larger class of potentials employing a Harnack-type inequality; this, however, led to a lower bound that is exponentially small in the interval length. Consequently, the main achievement of the present paper is to establish a polynomial bound for a rather general class of compactly supported potentials. In this sense, this paper provides a first step to approach the conjecture put forward in \cite{KT20}.

\section{The model and results}
We consider the Schrödinger operator 
\begin{equation}
h_L=-\frac{\ud^2}{\ud x^2}+v
\end{equation}
with a non-negative potential $v \in L^{\infty}(\mathbb{R})\cap L^{1}(\mathbb{R})$ on the interval $I=(-L/2,+L/2)$ with length $L > 0$; this implies that we are working on the Hilbert space $L^2(I)$. Furthermore, in order to obtain a self-adjoint operator, we impose Neumann boundary conditions at $x=\pm L/2$, i.e., we assume the derivative to vanish at the boundary. This operator has purely discrete spectrum with eigenvalues $0 \leq \lambda_0(L) \leq \lambda_1(L) \leq ...$; the ground state eigenfunction is non-degenerate and shall be denoted by $\varphi^L_0 \in L^2(I)$. We also introduce $k_0(L):=\sqrt{\lambda_0(L)}$.

The main object of interest in this paper is the spectral gap defined via
\begin{equation}
\Gamma_v(L) := \lambda_1(L)- \lambda_0(L)\ .
\end{equation}
Since the ground state is non-degenerate \cite{LiebLoss}, one has $\Gamma_v(L) > 0$ for every value $L > 0$. Furthermore, for potentials that decay fast enough at infinity (i.e. at least quadratically), $\Gamma_v(L)$ converges to zero like $\sim L^{-2}$ as $L \rightarrow \infty$; this follows from an adaptation of [Theorem~2.1,\cite{KT20}] to the case of Neumann boundary conditions. This holds, in particular, for potentials $v$ of compact support. 

Now, in a first result we derive a lower bound to $\Gamma_v(L)$ which only depends on the asymptotic behaviour of $k_0(L)$. The important aspect here is that asymptotic behaviour of the gap is reduced to studying the asymptotic behaviour of the ground state energy which is generally more accessible.

\begin{lemma}\label{LBXXX}Consider $h_L$ with a symmetric bounded potential $v(x)=v(-x)$, $v \geq 0$, whose support is $[-b,+b]$ for some $0 < b < \infty$. Furthermore, $v$ shall be strictly monotonically increasing on $[-b,-\varepsilon]$ for some $0 < \varepsilon \leq b$ and such that $\inf_{x \in [-\varepsilon,0]} v(x) > \gamma$ for some $\gamma > 0$. Assume in addition that $b^2\|v\|_{\infty} < 1/2$. Then
	\begin{equation*}
	\Gamma_v(L) \geq (1-2b^2\|v\|_{\infty})^2 \cdot \frac{\pi^2}{L^2}\cdot \cos^2\left(k_0(L)(L/2-b)\right)
	\end{equation*}
	holds for all $L$ large enough.
\end{lemma}
\begin{proof} For convenience, we shift the problem and work on the interval $(0,L)$ instead. We set $\tilde{v}(x):=v(x-L/2)$. Note that the assumptions on the potential guarantee that the ground state attains its maximum at the boundary of $(0,L)$ and the minimum at $x=L/2$. In particular, from the eigenvalue equation and the fact that $\lambda_0(L) \rightarrow 0$ (compare with the proof of [Theorem~2.1,\cite{KT20}]), it follows that $\varphi^L_0$ has at most one turning point in $(-L/2,0)$ for all $L$ large enough.
	
	Now, the ground state eigenfunction $\varphi^L_0$ on $(0,L/2-b)$ is given by 
	\begin{equation*}
	\varphi^L_0(x)=A \cos(k_0 x)
	\end{equation*}
	with $A=\varphi^L_0(0)=\sup_{x \in I}\varphi^L_0(x) > 0$. On the other hand, using the eigenvalue equation and the symmetry of $\tilde{v}$ with respect to $x=L/2$, we obtain
	\begin{equation}\label{TurningPoint}
	\varphi^{\prime}(x)=\int_{x}^{0}(\lambda_0(L)-\tilde{v}(x))\varphi^L_0(x)\ \ud x\ , \quad x \in [-b,0]\ ,
	\end{equation}
	and this yields, for $x \in [-b,0]$,
	\begin{equation*}\begin{split}
	|\varphi^{\prime}(x)| &\leq b\|\lambda_0(L)-\tilde{v}\|_{\infty}\cdot \sup_{x \in [-b,0)}|\varphi^L_0(x)|\ , \\
	&= Ab\|\lambda_0(L)-\tilde{v}\|_{\infty}\cdot \cos\left(k_0(L)(L/2-b)\right)\ ,\\ 
	& \leq 2Ab\|\tilde{v}\|_{\infty}\cdot \cos\left(k_0(L)(L/2-b)\right)\ ,
	\end{split}
	\end{equation*}
	for all $L > 0$ large enough. This gives
	\begin{equation*}
	\inf_{x \in I}\varphi^L_0(x) \geq (1-2b^2\|\tilde{v}\|_{\infty})\cdot A\cos\left(k_0(L)(L/2-b)\right)\ .
	\end{equation*}
	Finally, [Theorem~1.4,\cite{KirschGap}] yields (taking into account that the ground state eigenfunction of the Neumann Laplacian with zero potential is the constant function $1/\sqrt{L}$)
	\begin{equation*}
	\Gamma_v(L) \geq \left(\frac{\inf_{x \in I}\varphi^L_0(x)}{\sup_{x \in I}\varphi^L_0(x)}\right)^2 \cdot \frac{\pi^2}{L^2} \ ,
	\end{equation*}
	and hence the statement follows readily; see also \cite{KernerLowerBound} for a related application of [Theorem~1.4,\cite{KirschGap}].
\end{proof}
In a next step we need to investigate the asymptotic behaviour of $k_0(L)$ as $L \rightarrow \infty$. For the proof we introduce the characteristic function $\eins_A(\cdot)$ of a measurable set $A \subset \mathbb{R}$.
\begin{prop}\label{PropConvergence} Consider $h_L$ with a symmetric bounded potential $v(x)=v(-x)$, $v \geq 0$, whose support is $[-b,+b]$ for some $0 < b < \infty$. Furthermore, $v$ shall be strictly monotonically increasing on $[-b,-\varepsilon]$ for some $0 < \varepsilon \leq b$ and such that $\inf_{x \in [-\varepsilon,0]} v(x) > \gamma$ for some $\gamma > 0$. Then, there exists a constant $\delta > 0$ such that
	\begin{equation}\label{EqI}
\frac{\delta}{L} \leq 	\left|\frac{\pi}{2}-k_0(L)L\left(\frac{1}{2}-\frac{b}{L}\right)\right|
	\end{equation}
	for all $L$ large enough.
\end{prop} 
\begin{proof} The idea is to compare $k_0(L)$ with the square root of the lowest eigenvalue for the Laplacian with a step-potential $\tilde{v}(x):=\tilde{v} \cdot \eins_{[-c,+c]}(x)$ with suitable $\tilde{v}, c > 0$. The square root of this eigenvalue shall be denoted by $\tilde{k}_0(L)$. 
	
	In a first step we realize that, adapting the proof of [Proposition~2.9,\cite{KT20}], the quantisation condition for the step-potential with Neumann boundary conditions reads
	\begin{equation}\label{QCond}
	\frac{M_0}{\omega_0(L)}\tanh(M_0l_2)=\tan(\omega_0(L)l_1)
	\end{equation}
	where $\omega_0(L)L:=\tilde{k}_0(L)$, $M_0:=\sqrt{L^2\tilde{v}-\omega^2_0(L)}$, $l_1:=\frac{1}{2}-\frac{c}{L}$ and $l_2:=\frac{c}{L}$.
	
	
	On the other hand, in order to prove \eqref{EqI} we choose $c:=b$ and $\tilde{v}=\|v\|_{\infty}$. Hence, $\tilde{k}_0(L) \geq k_0(L)$ and it is therefore enough to prove \eqref{EqI} with $\tilde{k}_0(L)$ instead. To do this, we start with the estimate
	\begin{equation*}\begin{split}
	\tan(\omega_0(L)l_1) &\geq \frac{1}{2\cos(\omega_0(L)l_1)}\ , \\
	&\geq \frac{1}{2(-\omega_0(L)l_1+\pi/2)}\ ,
	\end{split}
	\end{equation*}
	for $\omega_0(L)l_1 < \pi/2$ close enough to $\pi/2$. Furthermore, setting $A:=M_0l_1\tanh(M_0l_2)$ and via the relation
	\begin{equation*}\begin{split}
	\frac{1}{2(-x+\pi/2)}=\frac{A}{x}
	\end{split}
	\end{equation*}
	we obtain 
	\begin{equation}\begin{split}
	\omega_0(L)l_1 &\leq \frac{\pi A}{(1+2A)}\ , \\
	&=\frac{\pi }{2 (1+\frac{1}{2A})}\ .
	\end{split}
	\end{equation}
This now yields, for all $L$ large enough,
\begin{equation*}\begin{split}
\left|\frac{\pi}{2}-\omega_0(L)l_1\right|&\geq\frac{\pi}{2}\left|\frac{1}{2A(1+\frac{1}{2A})}\right|\ , \\
& \geq \frac{\pi}{8A}\ , \\
&\geq \frac{\delta}{L}\ ,
\end{split}
\end{equation*}
for some constant $\delta > 0$.
\end{proof}
Combining Lemma~\ref{LBXXX} and Proposition~\ref{PropConvergence} then yields the main statement of the paper.
\begin{theorem}[Lower bound spectral gap]\label{PolynBound} Consider $h_L$ with a symmetric bounded potential $v(x)=v(-x)$, $v \geq 0$, whose support is $[-b,+b]$ for some $0 < b < \infty$. Furthermore, $v$ shall be strictly monotonically increasing on $[-b,-\varepsilon]$ for some $0 < \varepsilon \leq b$ and such that $\inf_{x \in [-\varepsilon,0]} v(x) > \gamma$ for some $\gamma > 0$. Assume in addition that $b^2\|v\|_{\infty} < 1/2$. Then, there exists a constant $\beta > 0$ such that
	\begin{equation*}
	\Gamma_v(L) \geq \frac{\beta}{L^4}
	\end{equation*}
	for all $L$ large enough.
\end{theorem}
\begin{proof} It remains to find a lower bound for the term 
	\begin{equation*}
	\cos^2\left(k_0(L)(L/2-b)\right)=\sin^2\left(\frac{\pi}{2}+k_0(L)(L/2-b)\right)
	\end{equation*}
	for large $L > 0$. Since $\sin(x) \geq -\frac{1}{2}x+\frac{\pi}{2}$ in a small neighbourhood around $x=\pi$, one obtains
	\begin{equation*}\begin{split}
	\left|\sin\left(\frac{\pi}{2}+k_0(L)(L/2-b)\right)\right| &\geq \left|-\frac{1}{2}\left(\frac{\pi}{2}+k_0(L)(L/2-b)\right) + \frac{\pi}{2}\right| \\
	&= \left|\frac{1}{2}\left(\pi-\frac{\pi}{2}-k_0(L)(L/2-b)\right)\right| \ 
	\end{split}
	\end{equation*}
	for all $x$ close enough to $\pi$. Taking the square one obtains
	\begin{equation*}
	\sin^2 \left(\frac{\pi}{2}+k_0(L)(L/2-b)\right)\geq\frac{1}{4} \left(\frac{\pi}{2}-k_0(L)L\left(\frac{1}{2}-\frac{b}{L}\right)\right)^2
	\end{equation*}
	The statement then follows with Proposition~\ref{PropConvergence}.
\end{proof}

\vspace*{0.5cm}

\subsection*{Acknowledgement}{}JK would like to thank M.~Täufer for helpful discussions.

\vspace*{0.5cm}

{\small
\bibliographystyle{amsalpha}
\bibliography{Literature}}

\end{document}